\documentclass[a4paper,11pt]{article}
\usepackage[hidelinks]{hyperref}
\hypersetup{
	colorlinks=true,   
	linktoc=all,        
	linkcolor=red,     
	citecolor=blue,}     
\usepackage{lineno}
\usepackage{tikz}
\usepackage{graphicx}
\usepackage{subfigure}
\usepackage{amssymb}
\usepackage{latexsym,bm}
\usepackage{multicol}
\usepackage{indentfirst}
\usepackage{amssymb,amsfonts}
\usepackage{amsmath}
\usepackage{amsthm}
\usepackage{setspace}
\usepackage{hyperref}
\usepackage{float} 
\usepackage{tikz}
\newcommand{\ignore}[1]{}

\newtheorem{theorem}{Theorem}[section]
\newtheorem{lemma}[theorem]{Lemma}
\newtheorem{claim}[theorem]{Claim}
\newtheorem{fact}[theorem]{Fact}
\newtheorem{corollary}[theorem]{Corollary}

\newtheorem{conjecture}[theorem]{Conjecture}

\newtheorem{observation}[theorem]{Observation}

\begin{document}
	\date{}
	\begin{spacing}{1.03}
		\title{{A step towards the Erd\H{o}s–Rogers problem 
        }}
		\author{Longma Du,\footnote{School of Mathematics, Shandong University, Jinan, 250100, P.~R.~China. Email: {\tt 202520303@mail.sdu.edu.cn}. }
       \;\; \;  Xinyu Hu,\footnote{Data Science Institute, Shandong University, Jinan, 250100, P.~R.~China. Email: {\tt huxinyu@sdu.edu.cn}. Supported by National Postdoctoral Fellowship Program (C-tier) (GZC20252005).}
			\;\; \; Ruilong Liu,\footnote{School of Mathematics, Shandong University, Jinan, 250100, P.~R.~China. Email: {\tt liuruilong@mail.sdu.edu.cn}. }
\;\; \; Guanghui Wang \footnote{
School of Mathematics, Shandong University, Jinan, 250100, P. R. China. Email: {\tt ghwang@sdu.edu.cn}. Supported by the Natural Science Foundation of China (12231018) and State Key Laboratory of Cryptography and Digital Economy Security.}
}
\maketitle	
		
\begin{abstract}
For $2\le k\le t<s$, the Erd\H{o}s-Rogers function $f^{(k)}_{t,s}(N)$ denotes the largest $m$ such that every $K^{(k)}_s$-free $k$-graph on $N$ vertices contains a $K^{(k)}_t$-free induced subgraph on $m$ vertices. 
Mubayi and Suk (\textit{J. London Math. Soc. 2018})  conjectured that $f^{(k)}_{k+1,k+2}(N)=(\log_{(k-2)}N)^{\Theta(1)}$ for $k\ge 4$, where $\log_{(i)}$ denotes the $i$-fold iterated logarithm. This is equivalent to the statement that $f^{(k)}_{k+1,s}(N)=(\log_{(k-2)}N)^{\Theta(1)}$ for every $s\ge k+2$.

In this paper, we introduce multi‑color patterns into a random construction of a $2$-graph to build a $4$-graph, and  for the first time, combine them with multi‑layer extremum structures to prove  that $f^{(4)}_{5,s}(N)=(\log \log N)^{\Theta(1)}$ for every $s\ge 11$. More generally, using a variant of the Erd\H{o}s-Hajnal stepping-up lemma, we also establish that $f^{(k)}_{k+1,s}(N)=(\log_{(k-2)}N)^{\Theta(1)}$ for every $s\ge k+7$.

\medskip

\end{abstract}
	
\section{Introduction}
A $k$-uniform hypergraph $H$ ($k$-graph for short) is a pair consisting of a set of vertices $V(H)$ and a collection of $k$-element subsets of $V(H)$. Let $K_n^{(k)}$ be the complete $k$-graph on $n$ vertices. The off-diagonal Ramsey number \cite{R} $r_k(s,n)$, which has been extensively studied when $k$ and $s$ are fixed, is the minimum $N$ such that every red/blue coloring of the edges of $K_N^{(k)}$ results in a monochromatic red copy of $K_s^{(k)}$ or a monochromatic blue copy of $K_n^{(k)}$. All logarithms are to base $2$ unless otherwise stated.

The off-diagonal Ramsey number has been extensively studied since 1935 \cite{E-S-1}, with many classic results \cite{A-K-S-1,A-K-S-2,B-K-2,C-G-E,E-H-Con,E-R-2,K-1,L-R-Z,M-S-3,M-S-4,S-1}.
In recent years, there have been many breakthroughs regarding off-diagonal Ramsey numbers, especially in graphs \cite{C-J-M-S,C-F-S-3,H-H-K-P,H-M-S,M-S-X,M-V-2}. Indeed, there have been significant recent breakthroughs in diagonal Ramsey numbers as well (see, e.g., \cite{C-G-M-S,G-N-N-W}). This work, however, focuses on the off-diagonal case, and diagonal Ramsey numbers are therefore not covered.

Campos, Jenssen, Michelen and Sahasrabudhe \cite{C-J-M-S} showed that $r_2(3,n)\ge (\frac{1}{3}+o(1))\frac{n^2}{\log n}$, which improved the results of Fiz Pontiveros, Griffiths and Morris \cite{P-G-M}. 
Shortly afterwards, Hefty, Horn, King and Pfender \cite{H-H-K-P} improved the lower bound $r_2(3,n)\ge (\frac{1}{2}+o(1))\frac{n^2}{\log n}$. Generally, the previous best bounds \cite{A-K-S-1,B-K-2,L-R-Z,SP-1} for $r_2(s,n)$ with fixed $s\ge 3$ are $\tilde{\Omega}(n^{\frac{s+1}{2}})\le r_2(s,n)\le  \tilde{O}(n^{s-1})$. A breakthrough result of Mattheus and Verstra\"{e}te \cite{M-V-2} showed that $r_2(4,n)\ge\Omega(\frac{n^3}{\log^4n})$.
 In particular, Erd\H{o}s \cite{C-G-E} conjectured that $r_2(s,n)= {\tilde{\Theta}}(n^{s-1})$, where $\tilde{O}$, $\tilde{\Omega}$, and $\tilde{\Theta}$ suppress factors of the form $(\log n)^{\Theta(1)}$.

For $3$-graphs, Conlon, Fox and Sudakov \cite{C-F-S-3} obtained that for every $s\ge4$, $2^{\Omega(n\log n)}\le r_3(s,n)\le 2^{O(n^{s-2}\log n)}$, which improve the upper bound by Erd\H{o}s and Rado \cite{E-R-2} and the lower bound by
Erd\H{o}s and Hajnal \cite{E-H-Con}. For $k\ge 4$, a fundamental and important conjecture about $r_k(s,n)$ was proposed by Erd\H{o}s and Hajnal \cite{E-H-Con}, who conjectured that $r_k(s,n)=twr_{k-1} (n^{\Theta(1)})$ for each $k\ge4$ and $s\ge k+1$, where the tower function $twr_{k}(x)$ is defined by $twr_{1}(x)=x$ and $twr_{i+1}(x)=2^{twr_{i}(x)}$. Up to now, the conjecture has been almost proven by Mubayi and Suk \cite{M-S-3,M-S-4}, as well as Conlon, Fox and Sudakov, and others, except for $s=k+1$. In particular, Mubayi and Suk \cite{M-S-3} used a variant of the classic Erd\H{o}s-Hajnal stepping-up lemma to prove that in order to completely solve the conjecture, one only needs to prove that $r_4(5,n)=2^{2^{n^{\Omega(1)}}}$.

A \textit{$t$-independent set} in a $k$-graph $G$ is a vertex subset that contains no copy of $K^{(k)}_t$. If $t=k$, then it is just an independent set. Let $\alpha_t(G)$ denote the size of the largest $t$-independent set in $G$. For $k\le t<s<N$, the \textit{Erd\H{o}s-Rogers function} $f^{(k)}_{t,s}(N)$ denotes the largest $m$ such that every $K^{(k)}_s$-free $k$-graph on $N$ vertices contains a $K^{(k)}_t$-free induced subgraph on $m$ vertices. 
Namely, $$f^{(k)}_{t,s}(N)=\min\{\alpha_t(G): |V(G)|=N~\text{and}~G~\text{is}~K^{(k)}_s\text{-free}\}.$$Clearly, $f^{(k)}_{t',s}(N)\le f^{(k)}_{t,s}(N) $ if $t'\le  t$ and $f^{(k)}_{t,s}(N)\le f^{(k)}_{t,s'}(N) $ if $s'\le s$.

\medskip

The Erd\H{o}s-Rogers function has been extensively studied in the last 60 years (see, e.g., \cite{ak,bh,D-R-R,dr,E-R-1,GJ20,H-L,Janzer,kri-1,kri-2,Sud-1,Sud-2,W-1}). We refer the reader to Mubayi and Suk \cite{M-S-1} for a nice survey on this topic. 

The problem of determining $f^{(k)}_{k,s}(N)$ is equivalent to determining the Ramsey number $r_k(s,n)$. Formally,
\begin{align}\label{R-EG}
r_k(s,n)=\min\{N: f^{(k)}_{k,s}(N)\ge n\}.
\end{align}

Note that the case $t=k$ recovers the usual Ramsey problem, as in (\ref{R-EG}). Thus, we focus on $f^{(k)}_{t,s}(N)$ with $t>k$. For $k=2$, the first lower bound was given by Bollob\'{a}s and Hind \cite{bh}, who showed that $f^{(2)}_{t,s}(N)\ge N^{1/(s-t+1)}$. Krivelevich \cite{kri-1,kri-2} improved these lower bounds by a small logarithmic factor and gave the general upper bound $f^{(2)}_{t,s}(N)\le O(N^{\frac{t}{s+1}}(\log N)^{\frac{1}{t-1}})$. Later, the lower bound was
significantly improved by Sudakov \cite{Sud-1,Sud-2} for every $s\ge t+2$. In particular, for $f^{(2)}_{t,t+1}(N)$, the presently best lower and upper bounds are as follows:
For each $t\ge3$,
$\Omega(N^{1/2}(\frac{\log N}{\log \log N})^{1/2})<f^{(2)}_{t,t+1}(N)<O(N^{1/2}\log N)$,
where the lower bound was obtained by Dudek and Mubayi \cite{D-M} and the upper bound was recently obtained by Mubayi and Verstra\"{e}te \cite{M-V-1}. For $f^{(2)}_{t,t+2}(N)$, Sudakov \cite{Sud-1} showed $f^{(2)}_{t,t+2}(N)\ge N^{\frac{3t-4}{6t-6}}(\log N)^{\Omega(1)}$, which is the currently best lower bound. However, the presently best upper bound is $f^{(2)}_{t,t+2}(N)\le O(N^{\frac{2t-3}{4t-5}}(\log N)^3)$, which was recently obtained by Janzer and Sudakov \cite{Janzer}.

The problem of estimating the Erd\H{o}s-Rogers function for $t\ge k\ge3$ seems much harder. Dudek and Mubayi \cite{D-M} proved that $\Omega((\log_{(k-2)} N)^{1/4+o(1)})\le f^{(k)}_{t,t+1}(N)\le O((\log N)^{1/(k-2)})$ by showing \begin{align}\label{ditui}
    f^{(k-1)}_{t-1,s-1}(\lfloor(\log N)^{1/(k-1)}\rfloor)\le f^{(k)}_{t,s}(N)
\end{align}
 for $3\le k\le t<s$, where $\log_{(0)}(N)=N$ and as usual, $\log_{(i+1)}N=\log (\log _{(i)}N)$. An interesting question is to determine the exponent of $\log N$ in $f^{(3)}_{t,t+1}(N)$. Conlon, Fox and Sudakov \cite{C-F-S-2} improved that $f^{(k)}_{t,t+1}(N)\ge \Omega((\log_{(k-2)} N)^{1/3+o(1)})$ by showing $f^{(3)}_{t,t+1}(N)\ge \Omega((\log N)^{1/3+o(1)})$, and they also proposed to close the gap between the upper and lower bounds. A first major step toward closing this gap was obtained by Mubayi and Suk \cite[Theorem 3.2]{M-S-5}, who showed that $f^{(k)}_{k+1,k+2}(N)\le O(\log_{(k-13)} N)$ for $k\ge 14$. They proposed a conjecture: $f^{(k)}_{k+1,k+2}(N)=(\log_{(k-2)}N)^{\Theta(1)}$ for each integer $k\ge 4$, which is equivalent to the following conjecture once we note that 
\begin{align}\label{low-bound}
    (\log_{(k-2)}N)^{\Omega(1)}\le f^{(k)}_{k+1,s}(N)\le f^{(k)}_{k+1,k+2}(N)
\end{align}
for $k\ge 3$ and $s\ge k+2$. Here, the first inequality follows from (\ref{ditui}) and $f^{(2)}_{3,s}(N)\ge N^{1/(s-2)}$, which was shown by Bollob\'{a}s and Hind \cite{bh}.

\begin{conjecture}[Rephrasing of a conjecture of Mubayi and Suk  in \cite{M-S-5}]
\label{con-1}
For each integer $k\ge 4$ and $s\ge k+2$, $$f^{(k)}_{k+1,s}(N)=(\log_{(k-2)}N)^{\Theta(1)}.$$
\end{conjecture}

Recently, Fan, Hu, Lin and Lu \cite{F-H-L-L} showed that  $f^{(k)}_{k+1,s}(N)\le 
O( \log_{(k-3)} N)$
for $k \ge 5$ and $s\ge k+2$, which together with (\ref{low-bound}) implies that there is only one `$\log$' gap to solve Conjecture \ref{con-1}. Moreover, they proved a variant of the Erd\H{o}s-Hajnal stepping-up lemma to obtain that in order to solve Conjecture \ref{con-1}, it suffices to show, for  every  $s\ge 6$, that
\begin{align}\label{con-2}
    f^{(4)}_{5,s}(N)=(\log\log N)^{O(1)}.
\end{align}
To prove (\ref{con-2}), one must construct a $K^{(4)}_s$-free $4$-graph $H$ of order $2^{2^{\Theta(n)}}$ with $\alpha_5(H)<\text{poly}(n)$.

In this paper, we first provide a probabilistic method to find a red/blue/green coloring $\phi$ of the pairs of $[2^{\Theta(n)}]$ that satisfies a  certain ``good'' property (see Lemma \ref{phi-11}). Based on  $\phi$, we then construct a $4$-graph $H$ on  $2^{2^{\Theta(n)}}$ vertices , which is $K^{(4)}_s$-free for every $s\ge 11$ and satisfies $\alpha_5(H)<2^{11}n^{11}+1$,  by analyzing the $11$-layer local maxima sequences in the stepping-up construction. This yields a new upper bound $f^{(4)}_{5,s}(N)\le (\log\log N)^{O(1)}$ for  every $s\ge 11$. Recall that $f^{(4)}_{5,s}(N)\ge (\log\log N)^{\Omega(1)}$ by (\ref{low-bound}),  we can obtain the following theorem.
\begin{theorem}\label{main-results}
  For  every $s\ge 11$, we have $f^{(4)}_{5,s}(N)= (\log\log N)^{\Theta(1)}$.
\end{theorem}

More generally, by applying a variant of the Erd\H{o}s-Hajnal stepping-up lemma, we obtain the following corollary.

\begin{corollary}\label{main-results-2}
For each $k\ge5$ and $s\ge k+7$, we have $f^{(k)}_{k+1,s}(N)= (\log_{(k-2)} N)^{\Theta(1)}$. 
\end{corollary}

Together with Theorem~\ref{main-results}, this confirms Conjecture~\ref{con-1} for all $k\ge 4$ and $s\ge k+7$.

\section{Properties of the stepping-up technique}\label{sec2}
In this paper, we will apply several variants of the Erd\H{o}s-Hajnal stepping-up lemma. We shall use the following notation and definitions unless otherwise stated. For integers $m$ and $n$, let $[n]=\{1,2,\ldots,n\}$ and $[m,n]=\{m,m+1,\ldots,n\}$.

Given a natural number $D$, let $V = \{0, 1, \ldots, 2^D - 1\}$. Then for any $v \in V$, write
$v = \sum_{i=0}^{D-1} v(i)2^i$ where  $v(i) \in \{0, 1\}$ for each  $i$. 
For any $u \neq v$, let $\delta(u, v)$ denote the largest $i \in \{0, 1, \ldots, D - 1\}$ such that $u(i) \neq v(i)$. We always use $\langle v_1, v_2,\ldots, v_r\rangle$ to denote an ordered set with $v_1< v_2<\cdots<v_r$. Given any set $S=\langle v_1, v_2,\ldots, v_r\rangle$, we always write $\delta_i=\delta(v_i,v_{i+1})$ for $i\in[r-1]$, and $\delta(S)=(\delta_1,\ldots,\delta_{r-1})=(\delta_i)_{i=1}^{r-1}$.  Let $(S)_i^j=\langle v_i,\dots,v_j\rangle\subset S$.
For convenience, if inequalities are known between consecutive $\delta$’s, this will be
indicated in the sequence by replacing the comma with the respective sign. For instance, suppose that $S=\langle v_1,\ldots, v_5\rangle$ and $\delta_1< \delta_2 >\delta_3< \delta_4$. Then, since $\delta(v_1,v_2,v_3,v_4)=(\delta_1,\delta_2,\delta_3)$
has $\delta_1< \delta_2 >\delta_3$, we write
$\delta(v_1,v_2,v_3,v_4)=(\delta_1<\delta_2>\delta_3)$.
Similarly, if not all inequalities are known, as in $\delta(v_1,v_2,v_4,v_5)$, we write $\delta(v_1,v_2,v_4,v_5)=(\delta_1<\delta_2~,~\delta_4)$.

We say that $\delta_i$ is a \textit{local minimum} if
$\delta_{i-1}>\delta_i<\delta_{i+1}$, a \textit{local maximum} if $\delta_{i-1}<\delta_i>\delta_{i+1}$, and a \textit{local extremum} if it is either a local minimum or a local maximum. We call $\delta_i$ a \textit{local monotone} if $\delta_{i-1}<\delta_i<\delta_{i+1}$ or $\delta_{i-1}>\delta_i>\delta_{i+1}$. We say $\delta_1,\ldots,\delta_{r-1}$
form a monotone sequence if $\delta_1<\cdots<\delta_{r-1}$ (monotonically increasing) or $\delta_1>\cdots>\delta_{r-1}$
(monotonically decreasing), i.e., there is no local extremum.

For any set $S=\langle v_1,\ldots,v_r \rangle$ where $r\ge 4$, let $m(S)$ and $n(S)$ denote the number of local extrema and local monotones in the sequence $\{\delta_1,\ldots, \delta_{r-1}\}$, respectively. Since for $i\in [2,r-2]$, each $\delta_i$ is either a local extremum or a local monotone, we have $m(S) + n(S)=r-3$.

Note that we have the following stepping-up properties, see \cite{G-R-S-1}.\medskip

\textbf{Property I.} For every triple $u < v < w$, $\delta(u,v) \neq \delta(v,w)$.
\medskip

\textbf{Property II.} For $v_1 < \cdots < v_r$, $\delta(v_1,v_r) = \max_{1 \leq j \leq r-1} \delta(v_j,v_{j+1})$.

\begin{fact}\label{fact}
    Let $\{\delta_{i_j}\}_{j=1}^{s}\subset \{\delta_t\}_{t=1}^{r-1}$. If $\delta_{i_j}\neq\delta_{i_{j+1}}$ for all $j$, then any non-monotone sequence $\{\delta_{i_j}\}_{j=1}^{s}$ contains a local extremum.
\end{fact}

We will also use the following stepping-up properties, which are easy consequences of Properties I and II, see \cite{C-F-S-1,F-H-L-L,H-L-L-W-1,M-S-3}.
\medskip

\textbf{Property III.} For $\delta(v_1,v_r) = \max_{1 \leq j \leq r-1} \delta(v_j,v_{j+1})$, there is a unique $\delta_i$ which achieves the maximum.
\medskip

\textbf{Property IV.} For every 4-tuple $v_1 < \cdots < v_4$, if $\delta(v_1,v_2) > \delta(v_2,v_3)$, then $\delta(v_1,v_2) \neq \delta(v_3,v_4)$. Note that if $\delta(v_1,v_2) < \delta(v_2,v_3)$, it is possible that $\delta(v_1,v_2) = \delta(v_3,v_4)$.
\medskip

\textbf{Property V.} For $v_1 < \cdots < v_r$, set $\delta_j = \delta(v_j,v_{j+1})$ for $j \in [r-1]$ and suppose that $\delta_1,\ldots,\delta_{r-1}$ form a monotone sequence. Then for every subset of $k$ vertices $v_{i_1},v_{i_2},\ldots,v_{i_k}$ where $v_{i_1} < \cdots < v_{i_k}$, $\delta(v_{i_1},v_{i_2}),\delta(v_{i_2},v_{i_3}),\ldots,\delta(v_{i_{k-1}},v_{i_k})$ form a monotone sequence. Moreover for every subset of $k-1$ such $\delta_j$'s, i.e., $\delta_{j_1},\delta_{j_2},\ldots,\delta_{j_{k-1}}$,
there are $k$ vertices $v_{i_1},\ldots,v_{i_k}$ such that $\delta(v_{i_t},v_{i_{t+1}}) = \delta_{j_t}$.

\section{The upper bound for $f^{(4)}_{5,s}(N)$ with $s\ge 11$}\label{u-b-f-11}
The upper bound for $f^{(4)}_{5,s}(N)$ follows by applying a variant of the Erd\H{o}s-Hajnal stepping-up lemma. This process begins with a graph possessing a specific property, the existence of which is guaranteed by a direct application of the probabilistic method.

\begin{lemma}\label{phi-11}
For every $n\ge 5$, there exists an absolute constant $c>0$ such that the following holds. There is a red/blue/green coloring $\phi$ of the pairs of $\{0,1,\ldots,\lfloor 2^{cn}\rfloor-1\}$ with the property that every $n$-set $A\subset \{0,1,\ldots,\lfloor 2^{cn}\rfloor-1\}$ contains a $4$-tuple $a_i<a_j<a_k<a_\ell$ satisfying $$\phi(a_i,a_j)=\phi(a_j,a_k)=\phi(a_k,a_\ell)=\text{red},~\phi(a_i,a_k)=\phi(a_j,a_\ell)=\text{blue},~\text{and}~\phi(a_i,a_\ell)=\text{green}.$$
\end{lemma}

\noindent\textit{Proof of Lemma~\ref{phi-11}.~}Set $D=\lfloor 2^{cn}\rfloor$, where $c>0$ is a sufficiently small absolute constant to be determined later.
Consider a random red/blue/green coloring $\phi$ of the pairs of $\{0,1,\ldots,D-1\}$ in which each pair is colored independently and uniformly from the three colors.

We call a 4-tuple $a_i,a_j,a_k,a_{\ell} \in \{0,1,\ldots,D-1\}$ with $a_i < a_j < a_k < a_{\ell}$ \emph{good} if 
$$\phi(a_i,a_j)=\phi(a_j,a_k)=\phi(a_k,a_\ell)=\text{red}, ~~\phi(a_i,a_k)=\phi(a_j,a_\ell)=\text{blue},~~\text{and}~~\phi(a_i,a_\ell)=\text{green},$$ and \emph{bad} otherwise. 
For a fixed $4$-tuple, exactly one pattern of its six pair-colors is good, and hence the probability that it is bad is $1-\frac{1}{3^6}=\frac{728}{729}$.

Now fix an $n$-subset $A\subset \{0,1,\ldots,D-1\}$. We estimate the probability that $A$ contains no good $4$-tuple. Note that there exists a partial Steiner $(n,4,2)$-system $S$\footnote{That is, a $4$-uniform hypergraph on an $n$-vertex set with at least $c'n^2$ edges, for some constant $c'>0$, in which every pair of vertices is contained in at most one edge.}\cite{E-H-2}, on $A$ with at least $c'n^2$ edges. Since any two distinct $4$-tuples in $S$ share at most one vertex, they share no pair, and therefore the corresponding bad events are independent. It follows that the probability that all $4$-tuples in $S$ are bad is at most $\left(\frac{728}{729}\right)^{c'n^2}$. Hence, the probability that $A$ contains no good $4$-tuple is also at most $\left(\frac{728}{729}\right)^{c'n^2}$.

Therefore, for sufficiently small absolute constant $c>0$, the expected number of $n$-subsets $A$ containing no good $4$-tuple is at most $\binom{D}{n}\left(\frac{728}{729}\right)^{c'n^2}<\frac{1}{2}$. Thus, by Markov's inequality, we conclude that there is a 3-coloring $\phi$ with the desired property. \hfill$\Box$
\medskip

Let $c>0$ be the constant from Lemma \ref{phi-11}, and let $U=\{0,1,\ldots,\lfloor2^{cn}\rfloor-1\}$ and $\phi: {{U}\choose{2}}\rightarrow \{\text{red}, \text{blue},\text{green}\}$ be a $3$-coloring of the pairs of $U$ satisfying the properties given in the lemma. Now, let $N=2^{\lfloor2^{cn}\rfloor}$ and $V(H)=\{0,1,\ldots,N-1\}$. For each $s\ge 11$, we shall use the coloring $\phi$ to produce a $K^{(4)}_s$-free $4$-graph $H$ on $V(H)$ with $\alpha_{5}(H)<{2^{11}n^{11}+1}$ as follows. For any $4$-tuple $e=\langle v_1,v_2,v_3,v_4 \rangle$ of $V(H)$, set $e\in E(H)$ if and only if one of the following holds:
\begin{enumerate}
\item[\textbf{($\Xi$)}] $\delta_1,\delta_2,\delta_3$ forms a monotone sequence and one of the following holds:
\begin{enumerate}
\item[\textbf{($\Xi_1$)}]$\phi(\delta_1,\delta_2)=\phi(\delta_2,\delta_3)=\text{red}$ and $\phi(\delta_1,\delta_3)=\text{blue}$.
\item[\textbf{($\Xi_2$)}]$\phi(\delta_1,\delta_2)=\text{red}$, $\phi(\delta_2,\delta_3)=\text{blue}$ and $\phi(\delta_1,\delta_3)=\text{green}$.
\item[\textbf{($\Xi_3$)}]$\phi(\delta_1,\delta_2)=\text{blue}$, $\phi(\delta_2,\delta_3)=\text{red}$ and $\phi(\delta_1,\delta_3)=\text{green}$.
\end{enumerate}

\item[\textbf{($\Lambda$)}] $\delta_1<\delta_2>\delta_3$.

\item[\textbf{($\Gamma$)}] $\delta_1>\delta_2<\delta_3$ and one of the following holds:
\begin{enumerate}
\item[\textbf{($\Gamma_1$)}] $\delta_1< \delta_3$, $\phi(\delta_1,\delta_2)=\phi(\delta_1,\delta_3)=\phi(\delta_2,\delta_3)=\text{red}$.
\item[\textbf{($\Gamma_2$)}]$\delta_1< \delta_3$, $\phi(\delta_1,\delta_2), \phi(\delta_1,\delta_3), \phi(\delta_2,\delta_3)\neq \text{red}$.
\item[\textbf{($\Gamma_3$)}]$\delta_1> \delta_3$, $\phi(\delta_1,\delta_2)= \phi(\delta_1,\delta_3)=\text{red}, \phi(\delta_2,\delta_3)\neq \text{red}$.
\item[\textbf{($\Gamma_4$)}]$\delta_1> \delta_3$, $\phi(\delta_1,\delta_3)=\text{blue}, \phi(\delta_2,\delta_3)= \text{red}$.
\item[\textbf{($\Gamma_5$)}]$\delta_1> \delta_3$, $\phi(\delta_1,\delta_2)\neq\text{blue}, \phi(\delta_1,\delta_3)=\text{green}, \phi(\delta_2,\delta_3)= \text{red}$.
\end{enumerate}
\end{enumerate}

\subsection{$H$ is $K_s^{(4)}$-free for every $s\ge 11$}

In this subsection, we show that $H$ is $K_s^{(4)}$-free for every $s \ge 11$. 
To see this, consider any set $P = \langle v_1, \ldots, v_s \rangle$ that induces a $K_s^{(4)}$ in $H$. We will prove that $s \le 10$. Recall that $\delta(P)=\{\delta_i\}_{i=1}^{s-1}$.

\begin{claim}\label{no-mono5}
There is no monotone subsequence $\{\delta'_{\ell}\}_{\ell=1}^{5}\subset \{\delta_i\}_{i=1}^{s-1}$ such that $\delta(u_1,\ldots,u_6)=(\delta'_{1},\ldots,\delta'_{5})$ for some $\{u_1,\ldots,u_6\}\subset\{v_1,\ldots,v_{s}\}$.
\end{claim}
\noindent\textit{Proof of Claim \ref{no-mono5}.~}Suppose, to the contrary, that $\{\delta'_{\ell}\}_{\ell=1}^5$
is a monotone increasing subsequence, without loss of generality, and that there exists $\{u_1,\ldots,u_6\}\subset\{v_i\}_{i=1}^{s} $ such that $\delta(u_1,\ldots,u_6)=(\delta'_1<\delta'_2<\delta'_3<\delta'_4<\delta'_5)$. Since $\delta(u_1,u_2,u_3,u_4)=(\delta'_1<\delta'_2<\delta'_3)$ and \textbf{($\Xi$)}, we have $\phi(\delta'_1,\delta'_3)\neq\text{red}$. Similarly, since $\delta(u_3,u_4,u_5,u_6)=(\delta'_3<\delta'_4<\delta'_5)$, we have $\phi(\delta'_3,\delta'_5)\neq\text{red}$. Therefore, $(u_1,u_2,u_4,u_6)\notin E(H[P])$ since $\delta(u_1,u_2,u_4,u_6)=(\delta'_1<\delta'_3<\delta'_5)$, and $\phi(\delta'_1,\delta'_3),\phi(\delta'_3,\delta'_5)\neq\text{red}$ and \textbf{($\Xi$)}, a contradiction.\hfill$\Box$

Let $\delta_k$ denote the unique largest element in $\delta(P)$. The uniqueness of $\delta_{k}$ follows from Property III. We first establish the following two claims.

\begin{claim}\label{no-left3}
There is no monotone decreasing subsequence $\{\delta'_{\ell}\}_{\ell=1}^{3}$ on the left of $\delta_k$ such that $\delta(u_1,u_2,u_3,u_4)=(\delta'_{1},\delta'_{2},\delta'_{3})$ for some $\{u_1,u_2,u_3,u_4\}\subset\{v_1,\ldots,v_{k}\}$.
\end{claim}
\noindent\textit{Proof of Claim \ref{no-left3}.~}Suppose, to the contrary, that $\delta(u_1,u_2,u_3,u_4)=(\delta'_{1}>\delta'_{2}>\delta'_{3})$ for some $\{u_1,u_2,u_3,u_4\}\subset\{v_1,\ldots,v_{k}\}$. If $\phi(\delta'_{1},\delta_k)=\text{red}$, then $(u_1,u_3,u_4,v_{k+1})\in E(H[P])$ implies $\phi(\delta'_{1},\delta'_{3})=\text{red}$ by \textbf{($\Gamma_1$)}. But $(u_1,u_2,u_3,u_4)\in E(H[P])$ implies $\phi(\delta'_{1},\delta'_{3})\neq\text{red}$ by \textbf{($\Xi$)}, a contradiction. Thus, $\phi(\delta'_{1},\delta_k)\neq\text{red}$. Note that $(u_1,u_2,u_3,v_{k+1})\in E(H[P])$ implies $\phi(\delta'_{1},\delta'_{2})\neq\text{red}$ and $\phi(\delta'_{2},\delta_{k})\neq\text{red}$   by \textbf{($\Gamma_2$)}, which together with $(u_2,u_3,u_4,v_{k+1})\in E(H[P])$ implies $\phi(\delta'_{2},\delta'_{3})\neq\text{red}$ again by \textbf{($\Gamma_2$)}. For $\phi(\delta'_{1},\delta'_{2}), \phi(\delta'_{2},\delta'_{3})\neq\text{red}$ and \textbf{($\Xi$)}, we have $(u_1,u_2,u_3,u_4)\notin E(H[P])$, again a contradiction.
\hfill$\Box$

\begin{claim}\label{no-right3}
There is no monotone increasing subsequence $\{\delta'_{\ell}\}_{\ell=1}^{3}$ on the right of $\delta_k$ such that $\delta(u_1,u_2,u_3,u_4)=(\delta'_{1},\delta'_{2},\delta'_{3})$ for some $\{u_1,u_2,u_3,u_4\}\subset\{v_{k+1},\ldots,v_{s}\}$.
\end{claim}
\noindent\textit{Proof of Claim \ref{no-right3}.~}Suppose, to the contrary, that $\delta(u_1,u_2,u_3,u_4)=(\delta'_{1}<\delta'_{2}<\delta'_{3})$ for some $\{u_1,u_2,u_3,u_4\}\subset\{v_{k+1},\ldots,v_{s}\}$. If $\phi(\delta_k,\delta'_{3})\neq\text{red}$, then $(v_k,u_1,u_2,u_4)\in E(H[P])$ implies $\phi(\delta'_{1},\delta'_{3})=\text{red}$ by \textbf{($\Gamma_4$)} or \textbf{($\Gamma_5$)}. But $(u_1,u_2,u_3,u_4)\in E(H[P])$ implies $\phi(\delta'_{1},\delta'_{3})\neq\text{red}$ by \textbf{($\Xi$)}, a contradiction. Thus, $\phi(\delta_k,\delta'_{3})=\text{red}$. Note that $(v_k,u_2,u_3,u_4)\in E(H[P])$ implies $\phi(\delta'_{2},\delta'_{3})\neq\text{red}$ and $\phi(\delta_{k},\delta'_{2})=\text{red}$ by \textbf{($\Gamma_3$)}, which together with $(v_k,u_1,u_2,u_3)\in E(H[P])$ implies that $\phi(\delta'_{1},\delta'_{2})\neq\text{red}$ by noting \textbf{($\Gamma_3$)} again.
For $\phi(\delta'_{1},\delta'_{2}), \phi(\delta'_{2},\delta'_{3})\neq\text{red}$ and \textbf{($\Xi$)}, we have $(u_1,u_2,u_3,u_4)\notin E(H[P])$, again a contradiction.\hfill$\Box$

Recall that $\delta_k$ is the unique largest element in $\delta(P)$. We next show that $k\le 5$ and $s-k\le 5$, which together yield $s\le 10$, as desired. We only prove $k\le 5$ here. The bound $s-k\le 5$ follows similarly, using Claim \ref{no-left3} in place of Claim \ref{no-right3} and applying the edge existence condition.

Suppose, to the contrary, that $k\ge 6$; that is, there are at least 5 elements on the left of $\delta_k$ in $\delta(P)$. Let $\delta_j$ denote the unique largest element in $\{\delta_1,\ldots,\delta_{k-1}\}$.  The uniqueness of $\delta_{j}$ follows from Property III. Then, $\delta_j<\delta_k$. We now turn to the following cases.

\textbf{Case 1.~} 
$\phi(\delta_j,\delta_k)=\text{red}$. Then we claim that $k-j\le2$. Otherwise, $k-j\ge3$. Thus, we have $\delta(v_j,v_{j+1},v_{j+2},v_{j+3},v_{k+1})=(\delta_j>\delta_{j+1},\delta_{j+2}<\delta_k)$. Moreover, Claim \ref{no-left3} and Property I imply that $\delta_{j+1}<\delta_{j+2}$. 
Since $(v_j,v_{j+1},v_{j+2},v_{k+1})\in E(H[P])$, $\phi(\delta_j,\delta_k)=\text{red}$, and $\delta_j<\delta_k$, it follows from \textbf{($\Gamma_1$)} that $\phi(\delta_{j+1},\delta_k)=\text{red}$. However, $(v_{j+1},v_{j+2},v_{j+3},v_{k+1})\in E(H[P])$ implies that $\phi(\delta_{j+1},\delta_{k})\neq\text{red}$ by \textbf{($\Xi$)}, a contradiction.
Thus, $j\ge k-2\ge 4$. 

Let $\delta_{\ell}$ denote the unique largest element in $\{\delta_1,\ldots,\delta_{j-1}\}$. Then we have $\delta_{\ell}<\delta_j<\delta_k$. Since
$(v_{\ell},v_{\ell+1},v_{j+1},v_{k+1})\in E(H[P])$ and $\phi(\delta_j,\delta_k)=\text{red}$, either $\phi(\delta_{\ell},\delta_j)=\text{red}$ and $\phi(\delta_{\ell},\delta_k)=\text{blue}$ or $\phi(\delta_{\ell},\delta_j)=\text{blue}$ and $\phi(\delta_{\ell},\delta_k)=\text{green}$ by \textbf{($\Xi$)}. We proceed to consider the two possibilities separately and each results in a contradiction.

If $\phi(\delta_{\ell},\delta_j)=\text{red}$ and $\phi(\delta_{\ell},\delta_k)=\text{blue}$, then $j-\ell=1$. Otherwise, we can find $\delta_{j-1}$ between $\delta_{\ell}$ and $\delta_j$ such that $\delta(v_{\ell},v_{j-1},v_{j},v_{j+1},v_{k+1})=(\delta_{\ell}>\delta_{j-1}<\delta_{j}<\delta_{k})$. Since $(v_{\ell},v_{j-1},v_{j},v_{j+1})\in E(H[P])$ and $\delta_{\ell}<\delta_j$, it follows from \textbf{($\Gamma_1$)} that $\phi(\delta_{\ell},\delta_{j-1})=\text{red}$. But $(v_{\ell},v_{j-1},v_{j},v_{k+1})\in E(H[P])$ and $\delta_{\ell}<\delta_k$ imply $\phi(\delta_{\ell},\delta_{j-1})\neq\text{red}$ by \textbf{($\Gamma_2$)}, a contradiction. For $\ell=j-1$, we have $$\delta(v_{j-3},v_{j-2},v_{j-1},v_{j},v_{j+1},v_{k+1})=(\delta_{j-3},\delta_{j-2}<\delta_{j-1}<\delta_j<\delta_k),$$ and hence $\delta_{j-3}>\delta_{j-2}$ by Claim \ref{no-mono5}. Since $(v_{j-3},v_{j-2},v_{j},v_{k+1})\in E(H[P])$ and $\phi(\delta_{j-1},\delta_k)=\text{blue}$ from assumption, 
$\phi(\delta_{j-3},\delta_{j-1})=\text{red}$ and $\phi(\delta_{j-3},\delta_{k})=\text{green}$ by \textbf{($\Xi_2$)}. It follows from $(v_{j-3},v_{j-2},v_{j-1},v_{j})\in E(H[P])$, $\delta_{j-3}<\delta_{j-1}$, $\phi(\delta_{j-1},\delta_{j})=\text{red}$ and \textbf{($\Gamma_1$)} that  $\phi(\delta_{j-3},\delta_{j-2})=\text{red}$. However, $(v_{j-3},v_{j-2},v_{{j-1}},v_{k+1})\in E(H[P])$ implies $\phi(\delta_{j-3},\delta_{j-2})\neq\text{red}$ by \textbf{($\Gamma_2$)}, a contradiction.

If $\phi(\delta_{\ell},\delta_j)=\text{blue}$ and $\phi(\delta_{\ell},\delta_k)=\text{green}$, then $\ell=1$. Otherwise, $\delta(v_1,v_{2},v_{{\ell}+1},v_{k+1})=(\delta_1<\delta_{\ell}<\delta_k)$ and $\phi(\delta_{\ell},\delta_k)=\text{green}$ imply $(v_1,v_{2},v_{{\ell}+1},v_{k+1})\notin E(H[P])$, a contradiction. For $\ell=1$ and $j\ge 4$, we have $$\delta(v_{1},v_2,v_{3},v_{4},v_{j+1},v_{k+1})=(\delta_{1}>\delta_{2},\delta_{3}<\delta_j<\delta_{k}),$$ and hence $\delta_{2}<\delta_{3}$ by Claim \ref{no-left3}. Since $(v_{2},v_{3},v_{j+1},v_{k+1})\in E(H[P])$, it follows from \textbf{($\Xi$)} that $\phi(\delta_{2},\delta_j)\neq\text{green}$. Since $(v_{1},v_{3},v_{4},v_{j+1})\in E(H[P])$ and $\phi(\delta_1,\delta_j)=\text{blue}$, $\phi(\delta_{3},\delta_j)\neq\text{red}$ from \textbf{($\Gamma_2$)}. Combining these two facts about $\phi(\delta_2, \delta_j)$ and $\phi(\delta_{3}, \delta_j)$, we see from \textbf{($\Xi$)} that $(v_{2},v_{3},v_{4},v_{j+1})\notin E(H[P])$, a contradiction.
\medskip

\textbf{Case 2.~} 
$\phi(\delta_j,\delta_k)\neq\text{red}$. Then we claim $k-j\le 3$. Thus, $j\ge 3$. Otherwise, we have $k-j\ge 4$. Note that $$\delta(v_{j},v_{j+1},v_{j+2},v_{j+3},v_{j+4},v_{k+1})=(\delta_{j}>\delta_{j+1},\delta_{j+2},\delta_{j+3}<\delta_{k}),$$ and hence $\delta_{j+1}<\delta_{j+2}<\delta_{j+3}$ by Claim \ref{no-left3}. However, $\{\delta_{j},\delta_{j+1},\delta_{j+2},\delta_{j+3}\}$ then yields a contradiction by an argument analogous to the proof of Claim~\ref{no-right3}.
 
Note that $j\ge 3$ and we can consider $\delta(v_1,v_{2},v_{3},v_{j+1},v_{k+1})=(\delta_{1},\delta_{2}<\delta_j<\delta_{k})$. Since $(v_1,v_{2},v_{j+1},v_{k+1})\in E(H[P])$, it follows from \textbf{($\Xi_2$)} that $\phi(\delta_{1},\delta_j)=\text{red}$ and $\phi(\delta_{1},\delta_k)=\text{green}$
. If $\delta_1<\delta_{2}$, then applying \textbf{($\Xi$)} to the fact that 
$\phi(\delta_{1},\delta_{j})=\text{red}$ yields
$(v_1,v_{2},v_{3},v_{j+1})\notin E(H[P])$, a contradiction.
Hence, by Property I, we must have $\delta_1>\delta_{2}$. Given this, together with the conditions $(v_1,v_{2},v_{3},v_{j+1})\in E(H[P])$ 
and $\delta_j>\delta_1$, \textbf{($\Gamma_1$)} yields 
$\phi(\delta_{1},\delta_{2})=\text{red}$. 
This, combined with $\phi(\delta_{1},\delta_k)=\text{green}$, 
forces $(v_1,v_{2},v_{3},v_{k+1})\notin E(H[P])$, leading to another contradiction. This completes the proof of $k\le 5$.

\subsection{$\alpha_5(H)<2^{11}n^{11}+1$}
In this subsection, we show that $\alpha_5(H)<2^{11}n^{11}+1$. Suppose to the contrary that there are vertices $Q=\langle v_1,v_2,\cdots ,v_m\rangle$,  $m=2^{11}n^{11}+1$, which induce a $K^{(4)}_5$-independent set in $H$. Recall that $\delta_i=\delta(v_i,v_{i+1})$, and hence $\delta(Q)=\{\delta_1,\ldots, \delta_{m-1}\}$. 

\begin{lemma}\label{bad-11}
There is no monotone subsequence $\{\delta_{i_\ell}\}_{\ell=1}^n\subset \{\delta_j\}_{j=1}^{m-1}$ such that for any $a,b,c,d \in [n]$ with $a<b<c<d$, there exists $\{u_1,\cdots,u_5\}\subset\{v_1,\ldots,v_m\}$ such that $\delta(u_1,\ldots,u_5)=(\delta_{i_a},\delta_{i_b},\delta_{i_c},\delta_{i_d})$.
\end{lemma}
\noindent\textit{Proof of Lemma \ref{bad-11}.~}Suppose to the contrary that $\{\delta_{i_\ell}\}_{\ell=1}^n$ is such a monotone increasing subsequence without loss of generality. It follows from Lemma \ref{phi-11} that there is a $4$-tuple $\delta_{i_a},\delta_{i_b},\delta_{i_c},\delta_{i_d}$ with $\delta_{i_a}<\delta_{i_b}<\delta_{i_c}<\delta_{i_d}$ such that $$\phi(\delta_{i_a},\delta_{i_b})=\phi(\delta_{i_b},\delta_{i_c})=\phi(\delta_{i_c},\delta_{i_d})=\text{red}, ~~\phi(\delta_{i_a},\delta_{i_c})=\phi(\delta_{i_b},\delta_{i_d})=\text{blue}, ~~\phi(\delta_{i_a},\delta_{i_d})=\text{green}.$$ Since there exists $\{u_1,\cdots,u_5\}\subset \{v_1,\ldots,v_m\}$ such that $\delta(u_1,\ldots,u_5)=(\delta_{i_a},\delta_{i_b},\delta_{i_c},\delta_{i_d})$ from the assumption, $\{u_1,\ldots,u_5\}$ forms a copy of  $K^{(4)}_5$ from \textbf{($\Xi$)}, a contradiction.\hfill$\Box$
\medskip

Denote $\beta_i=\frac{m-1}{(2n)^i}$  for $i\in[0,11]$. For $t\in[11]$, 
we will greedily construct \textit{$t$-layer local maxima sequences} $\Delta^{(t)}$ such that $\Delta^{(t)}\subset\Delta^{(t-1)}$, starting with $\Delta^{(0)}=\delta(Q)$, and satisfy the following property. (We do not require that the elements of $\Delta^{(t)}$ be distinct.)

\begin{itemize}
    \item[($\ast$)] For two consecutive elements $\delta_a$, $\delta_b\in\Delta^{(t)}$, we have $\delta_x<\max\{\delta_a ,\delta_b\}$ for all $a<x<b$, and hence $\delta_a\neq\delta_b$.
\end{itemize}

For $t\ge1$, assume now that we have obtained $\Delta^{(t-1)}$ satisfying the desired property. We restrict our attention to $\Delta^{(t-1)}$ and we will find $\Delta^{(t)}$ to be the first $\beta_t$ local maxima (with respect to the sequence $\Delta^{(t-1)}$) as follows.  For convenience, we abbreviate ``with respect to " as ``w.r.t." in the following discussion. We claim first that there is no monotone consecutive subsequence of length $n$. Otherwise, suppose such a subsequence $Q'$ exists. Without loss of generality, assume $Q'$ is increasing. For any $\delta_a,\delta_b,\delta_c,\delta_d\in Q'$ with $a<b<c<d$. Then $\delta(v_{a},v_{a+1},v_{b+1},v_{c+1},v_{d+1})=(\delta_{a}<\delta_{b}<\delta_{c}<\delta_{d})$ by noting the first part of the property of $\Delta^{(t-1)}$ and Property II, which contradicts Lemma \ref{bad-11}. It follows from the second part of the property ($\ast$) of $\Delta^{(t-1)}$ and Fact \ref{fact} that we can set $\Delta^{(t)}$ to be the first $\beta_t$ local maxima (w.r.t. $\Delta^{(t-1)}$). Therefore, $\Delta^{(t)}\subset \Delta^{(t-1)}$. 

To show the  property ($\ast$) for $\Delta^{(t)}$,  we consider two consecutive elements $\delta_a$, $\delta_b\in\Delta^{(t)}$ and  we may assume that $\delta_{a},\delta_{i_1},\delta_{i_2},\cdots,\delta_{i_j},\delta_{b}$ are consecutive elements in $\Delta^{(t-1)}$. Note that $\delta_{a}$ and $\delta_{b}$ are consecutive local maxima (w.r.t. $\Delta^{(t-1)}$), we have $\delta_{i_\ell}<\max\{\delta_a, \delta_b\}$ for $\ell\in[j]$. Furthermore, it follows from the inductive hypothesis that $\delta_x<\max\{\delta_{i_\ell},\delta_{i_{\ell+1}}\}$ for all $i_\ell<x<i_{\ell+1}$ and $\ell\in[j-1]$, then $\delta_x<\max\{\delta_{i_\ell},\delta_{i_{\ell+1}}\}<\max\{\delta_{a},\delta_{b}\}$. Thus, $\delta_x<\max\{\delta_{a},\delta_{b}\}$ for all $a<x<b$. Moreover, Property I implies that $\delta_a\neq\delta_b$, as desired. Otherwise $\delta(v_a,v_b)=\delta_a=\delta_b=\delta(v_b,v_{b+1})$, a contradiction.

For $t\in[11]$ and $\delta_j\in\Delta^{(t)}\backslash\Delta^{(t+1)}$, where $\Delta^{(12)}=\emptyset$. Note that $\delta_j$ is a local maximum (w.r.t. $\Delta^{(t-1)}$), we always let $\delta_{j^-}$ and $\delta_{j^+}$ be the \textbf{closest} element to the left and right of $\delta_j$ in the sequence $\Delta^{(t-1)}$, respectively. In particular, $\delta_{j^-},\delta_{j^+}\in \Delta^{(t-1)}\setminus\Delta^{(t)}$. From the above greedy construction, we can obtain the following observation by repeatedly using ($\ast$).

\begin{observation}\label{observation-1}
For $t\in [11]$ and  $\delta_j\in \Delta^{(t)}\backslash\Delta^{(t+1)}$, 
we have $\delta_x<\delta_{j}$ for each $x\in[j^-,j^+]\backslash \{j\}$.
\end{observation}

Note that $|\Delta^{(11)}|=\beta_{11}=1$. We will find $K^{(4)}_5$ of \textbf{($\Gamma$, $\Lambda$)}-type  in $H[Q]$ by using the following claims, which contradicts our assumption that $H[Q]$ is $K^{(4)}_5$-free. 

\begin{claim}\label{blue} For any {$t\in [5,11]$}, let $\delta_{a}\in\Delta^{(t)}\backslash\Delta^{(t+1)}$ and $\delta_{b}=\delta_{a^+}$.
For $\delta_c\in\Delta^{(4)}$ and $a<c\le b$, we have $\phi(\delta_a,\delta_c)\neq\text{blue}$.
\end{claim}
\begin{figure}[htbp] 
\begin{center}

\tikzset{every picture/.style={line width=0.75pt}} 

\begin{tikzpicture}[x=0.75pt,y=0.75pt,yscale=-1,xscale=1]

\draw    (200,3109.5) -- (200,2958.5) ;
\draw [shift={(200,2956.5)}, rotate = 90] [color={rgb, 255:red, 0; green, 0; blue, 0 }  ][line width=0.75]    (10.93,-3.29) .. controls (6.95,-1.4) and (3.31,-0.3) .. (0,0) .. controls (3.31,0.3) and (6.95,1.4) .. (10.93,3.29)   ;
\draw  [fill={rgb, 255:red, 0; green, 0; blue, 0 }  ,fill opacity=1 ] (444.5,3030.52) .. controls (444.5,3028.85) and (445.85,3027.5) .. (447.52,3027.5) .. controls (449.18,3027.5) and (450.53,3028.85) .. (450.53,3030.52) .. controls (450.53,3032.18) and (449.18,3033.53) .. (447.52,3033.53) .. controls (445.85,3033.53) and (444.5,3032.18) .. (444.5,3030.52) -- cycle ;
\draw  [fill={rgb, 255:red, 0; green, 0; blue, 0 }  ,fill opacity=1 ] (379.5,3050.52) .. controls (379.5,3048.85) and (380.85,3047.5) .. (382.52,3047.5) .. controls (384.18,3047.5) and (385.53,3048.85) .. (385.53,3050.52) .. controls (385.53,3052.18) and (384.18,3053.53) .. (382.52,3053.53) .. controls (380.85,3053.53) and (379.5,3052.18) .. (379.5,3050.52) -- cycle ;
\draw  [fill={rgb, 255:red, 0; green, 0; blue, 0 }  ,fill opacity=1 ] (265.5,3048.52) .. controls (265.5,3046.85) and (266.85,3045.5) .. (268.52,3045.5) .. controls (270.18,3045.5) and (271.53,3046.85) .. (271.53,3048.52) .. controls (271.53,3050.18) and (270.18,3051.53) .. (268.52,3051.53) .. controls (266.85,3051.53) and (265.5,3050.18) .. (265.5,3048.52) -- cycle ;
\draw  [fill={rgb, 255:red, 0; green, 0; blue, 0 }  ,fill opacity=1 ] (284.5,3029.52) .. controls (284.5,3027.85) and (285.85,3026.5) .. (287.52,3026.5) .. controls (289.18,3026.5) and (290.53,3027.85) .. (290.53,3029.52) .. controls (290.53,3031.18) and (289.18,3032.53) .. (287.52,3032.53) .. controls (285.85,3032.53) and (284.5,3031.18) .. (284.5,3029.52) -- cycle ;
\draw  [fill={rgb, 255:red, 0; green, 0; blue, 0 }  ,fill opacity=1 ] (424.5,3013.52) .. controls (424.5,3011.85) and (425.85,3010.5) .. (427.52,3010.5) .. controls (429.18,3010.5) and (430.53,3011.85) .. (430.53,3013.52) .. controls (430.53,3015.18) and (429.18,3016.53) .. (427.52,3016.53) .. controls (425.85,3016.53) and (424.5,3015.18) .. (424.5,3013.52) -- cycle ;
\draw  [fill={rgb, 255:red, 0; green, 0; blue, 0 }  ,fill opacity=1 ] (458.5,2996.52) .. controls (458.5,2994.85) and (459.85,2993.5) .. (461.52,2993.5) .. controls (463.18,2993.5) and (464.53,2994.85) .. (464.53,2996.52) .. controls (464.53,2998.18) and (463.18,2999.53) .. (461.52,2999.53) .. controls (459.85,2999.53) and (458.5,2998.18) .. (458.5,2996.52) -- cycle ;
\draw  [fill={rgb, 255:red, 0; green, 0; blue, 0 }  ,fill opacity=1 ] (223.5,2990.52) .. controls (223.5,2988.85) and (224.85,2987.5) .. (226.52,2987.5) .. controls (228.18,2987.5) and (229.53,2988.85) .. (229.53,2990.52) .. controls (229.53,2992.18) and (228.18,2993.53) .. (226.52,2993.53) .. controls (224.85,2993.53) and (223.5,2992.18) .. (223.5,2990.52) -- cycle ;
\draw  [fill={rgb, 255:red, 0; green, 0; blue, 0 }  ,fill opacity=1 ] (245.5,2970.52) .. controls (245.5,2968.85) and (246.85,2967.5) .. (248.52,2967.5) .. controls (250.18,2967.5) and (251.53,2968.85) .. (251.53,2970.52) .. controls (251.53,2972.18) and (250.18,2973.53) .. (248.52,2973.53) .. controls (246.85,2973.53) and (245.5,2972.18) .. (245.5,2970.52) -- cycle ;
\draw  [fill={rgb, 255:red, 0; green, 0; blue, 0 }  ,fill opacity=1 ] (398.5,3073.52) .. controls (398.5,3071.85) and (399.85,3070.5) .. (401.52,3070.5) .. controls (403.18,3070.5) and (404.53,3071.85) .. (404.53,3073.52) .. controls (404.53,3075.18) and (403.18,3076.53) .. (401.52,3076.53) .. controls (399.85,3076.53) and (398.5,3075.18) .. (398.5,3073.52) -- cycle ;
\draw  [fill={rgb, 255:red, 0; green, 0; blue, 0 }  ,fill opacity=1 ] (316.5,3078.52) .. controls (316.5,3076.85) and (317.85,3075.5) .. (319.52,3075.5) .. controls (321.18,3075.5) and (322.53,3076.85) .. (322.53,3078.52) .. controls (322.53,3080.18) and (321.18,3081.53) .. (319.52,3081.53) .. controls (317.85,3081.53) and (316.5,3080.18) .. (316.5,3078.52) -- cycle ;
\draw  [fill={rgb, 255:red, 0; green, 0; blue, 0 }  ,fill opacity=1 ] (347.5,3094.52) .. controls (347.5,3092.85) and (348.85,3091.5) .. (350.52,3091.5) .. controls (352.18,3091.5) and (353.53,3092.85) .. (353.53,3094.52) .. controls (353.53,3096.18) and (352.18,3097.53) .. (350.52,3097.53) .. controls (348.85,3097.53) and (347.5,3096.18) .. (347.5,3094.52) -- cycle ;
\draw  [dash pattern={on 4.5pt off 4.5pt}]  (226.52,2990.52) -- (248.52,2970.52) ; 
\draw  [dash pattern={on 4.5pt off 4.5pt}]  (350.52,3094.52) -- (319.52,3078.52) ;
\draw  [dash pattern={on 4.5pt off 4.5pt}]  (401.52,3073.52) -- (382.52,3050.52) ;
\draw  [dash pattern={on 4.5pt off 4.5pt}]  (319.52,3078.52) -- (382.52,3050.52) ;
\draw  [dash pattern={on 4.5pt off 4.5pt}]  (382.52,3050.52) -- (287.52,3029.52) ;
\draw  [dash pattern={on 4.5pt off 4.5pt}]  (268.52,3048.52) -- (287.52,3029.52) ; 
\draw  [dash pattern={on 4.5pt off 4.5pt}]  (447.52,3030.52) -- (427.52,3013.52) ;
\draw  [dash pattern={on 4.5pt off 4.5pt}]  (287.52,3029.52) -- (427.52,3013.52) ;
\draw  [dash pattern={on 4.5pt off 4.5pt}]  (248.52,2970.52) -- (461.52,2996.52) ;
\draw    (200,3109.5) -- (517,3111.5) ;

\draw (239.65,2949.5) node [anchor=north west][inner sep=0.75pt]  [font=\small]  {$\delta _{a}$};

\draw (208.65,2987.5) node [anchor=north west][inner sep=0.75pt]  [font=\small]  {$\delta _{a-1}$};

\draw (453.65,2976.5) node [anchor=north west][inner sep=0.75pt]  [font=\small]  {$\delta _{b}$};

\draw (419.65,2993.5) node [anchor=north west][inner sep=0.75pt]  [font=\small]  {$\delta _{c}$};

\draw (278.65,3008.5) node [anchor=north west][inner sep=0.75pt]  [font=\small]  {$\delta _{d}$};

\draw (250.65,3045.5) node [anchor=north west][inner sep=0.75pt]  [font=\small]  {$\delta _{d-1}$};

\draw (453.65,3025.5) node [anchor=north west][inner sep=0.75pt]  [font=\small]  {$\delta _{c+1}$};

\draw (388.65,3043.5) node [anchor=north west][inner sep=0.75pt]  [font=\small]  {$\delta _{e}$};

\draw (297.65,3070.5) node [anchor=north west][inner sep=0.75pt]  [font=\small]  {$\delta _{f}$};

\draw (356.65,3088.5) node [anchor=north west][inner sep=0.75pt]  [font=\small]  {$\delta _{g}$};

\draw (406.65,3068.5) node [anchor=north west][inner sep=0.75pt]  [font=\small]  {$\delta _{e+1}$};

\end{tikzpicture}

\end{center}
\caption{A diagram illustrating the proof process for Claims \ref{blue} and \ref{green}. }
\label{figure:greedy}
\end{figure}
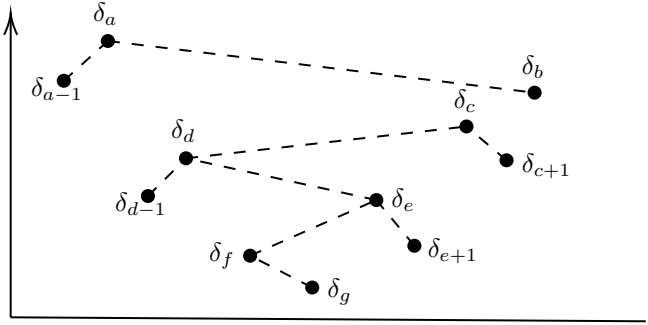
\noindent\textit{Proof of Claim \ref{blue}.~}Suppose, to the contrary, that there exists $\delta_c\in\Delta^{(4)}$ and $a<c\le b$ such that $\phi(\delta_a,\delta_c)=\text{blue}$. 
Consider $\delta_d=\delta_{c^-}\in\Delta^{(3)}$, $\delta_e=\delta_{d^+}\in\Delta^{(2)}$, $\delta_f=\delta_{e^-}\in\Delta^{(1)}$, and $\delta_g=\delta_{f^+}\in\Delta^{(0)}$. It follows from Observation \ref{observation-1} that $$a-1<a<d-1<d<f<g<e<e+1<c\le b.$$ In particular, if $c<b$, then $c<c+1<b$. (See Figure \ref{figure:greedy}).

Note that $$\delta(v_{a-1},v_{a},v_{d},v_{d+1},v_{c+1})=(\delta_{a-1}<\delta_a>\delta_d<\delta_c)~\text{and}~\delta_a>\delta_c$$ from Property II and Observation \ref{observation-1}. 
If $\phi(\delta_d,\delta_c)=\text{red}$, then it follows from $\phi(\delta_a,\delta_c)=\text{blue}$ and \textbf{($\Gamma_4$)} that 
$(v_{a-1},v_{d},v_{d+1},v_{c+1}),(v_{a},v_{d},v_{d+1},v_{c+1})\in E(H[Q])$. Furthermore, note that \textbf{($\Lambda$)}, we have that  $\{v_{a-1},v_{a},v_{d},v_{d+1},v_{c+1}\}$ forms a $K_5^{(4)}$, a contradiction. Hence, $\phi(\delta_d,\delta_c)\neq\text{red}$. Similarly, we have $\phi(\delta_e,\delta_c),\phi(\delta_f,\delta_c),\phi(\delta_g,\delta_c)\neq\text{red}$. 

It holds that $$\delta(v_{d},v_{e},v_{e+1},v_{c+1},v_{c+2})=(\delta_{d}>\delta_e<\delta_c>\delta_{c+1})~\text{and}~\delta_c>\delta_d$$ from Property II and Observation \ref{observation-1}, we claim  $\phi(\delta_d,\delta_e)=\text{red}$. Otherwise,  $(v_{d},v_{e},v_{e+1},v_{c+1})$, $(v_{d},v_{e},v_{e+1},v_{c+2})\in E(H[Q])$ follows from $\phi(\delta_d,\delta_c),\phi(\delta_e,\delta_c)\neq\text{red}$ and \textbf{($\Gamma_2$)}. Together with \textbf{($\Lambda$)}, we  obtain  $\{v_{d},v_{e},v_{e+1},v_{c+1},v_{c+2}\}$ that forms a $K_5^{(4)}$,  a contradiction.  Similarly, we have $\phi(\delta_d,\delta_f)=\phi(\delta_d,\delta_g)=\phi(\delta_f,\delta_g)=\text{red}$. 

If $\phi(\delta_f,\delta_e)\neq\text{red}$, then it follows from \textbf{($\Gamma_3$)} and \textbf{($\Lambda$)} that $\{v_{d-1},v_{d},v_{f},v_{f+1},v_{e+1}\}$ forms a $K_5^{(4)}$  by noting  $\phi(\delta_d,\delta_f)=\phi(\delta_d,\delta_e)=\text{red}$, 
$\delta(v_{d-1},v_{d},v_{f},v_{f+1},v_{e+1})=(\delta_{d-1}<\delta_d>\delta_f<\delta_{e})$ and $\delta_d>\delta_e$,  a contradiction. Hence, $\phi(\delta_f,\delta_e)=\text{red}$. Similarly, $\phi(\delta_g,\delta_e)=\text{red}$. 
Therefore, $\{v_{f},v_{g},v_{g+1},v_{e+1},v_{e+2}\}$ forms a $K_5^{(4)}$ by applying $\phi(\delta_f,\delta_g)=\phi(\delta_f,\delta_e)=\phi(\delta_g,\delta_e)=\text{red}$, \textbf{($\Gamma_1$)} and \textbf{($\Lambda$)}, which is a contradiction.  Consequently, $\phi(\delta_a,\delta_c)\neq\text{blue}$, as desired.\hfill$\Box$

\begin{claim}\label{green}
For any $t\in [9,11]$, let $\delta_{a}\in\Delta^{(t)}\backslash\Delta^{(t+1)}$ and $\delta_{b}=\delta_{a^+}$. For $\delta_c\in\Delta^{(8)}$ and $a<c\le b$, we have $\phi(\delta_a,\delta_c)\neq\text{green}$.
\end{claim}
\noindent\textit{Proof of Claim \ref{green}.~} Suppose, to the contrary, that there exists $\delta_c\in\Delta^{(8)}$ and $a<c\le b$ such that $\phi(\delta_a,\delta_c)=\text{green}$. Consider $\delta_d=\delta_{c^-}\in\Delta^{(7)}$, $\delta_e=\delta_{d^+}\in\Delta^{(6)}$, $\delta_f=\delta_{e^-}\in\Delta^{(5)}$ and $\delta_g=\delta_{f^+}\in\Delta^{(4)}$. Then, $$a-1<a<d<f<g<e<c\le b$$ from Observation \ref{observation-1}. (See Figure \ref{figure:greedy}). Recall that $\Delta^{(i)}\subset \Delta^{(4)}$ for each $i\in[5,11]$. By Claim \ref{blue}, we know that $\phi(\delta_a,\delta_d),\phi(\delta_a,\delta_e),\phi(\delta_a,\delta_f),\phi(\delta_a,\delta_g)\neq\text{blue}$. 

Note that $$\delta(v_{a-1},v_{a},v_{d},v_{d+1},v_{c+1})=(\delta_{a-1}<\delta_a>\delta_d<\delta_c)~\text{and}~\delta_a>\delta_c.$$ If $\phi(\delta_d,\delta_c)=\text{red}$, then it follows from \textbf{($\Gamma_5$)} that $(v_{a-1},v_{d},v_{d+1},v_{c+1}),(v_{a},v_{d},v_{d+1},v_{c+1})\in E(H[Q])$. Together with \textbf{($\Lambda$)}, we know that $\{v_{a-1},v_{a},v_{d},v_{d+1},v_{c+1}\}$ forms a $K_5^{(4)}$, a contradiction. Hence, $\phi(\delta_d,\delta_c)\neq\text{red}$. Similarly, $\phi(\delta_e,\delta_c),\phi(\delta_f,\delta_c),\phi(\delta_g,\delta_c)\neq\text{red}$. The argument then proceeds exactly as in the proof of Claim~\ref{blue}, leading to a contradiction.\hfill$\Box$
\medskip

Recall that $|\Delta^{(11)}|=\beta_{11}=1$, then we can choose $\delta_{a}\in\Delta^{(11)}$.
Let $\delta_{b}=\delta_{a^+}\in\Delta^{(10)}$, $\delta_c=\delta_{b^-}\in\Delta^{(9)}$ and $\delta_d=\delta_{c^+}\in\Delta^{(8)}$. Thus, $$a-1<a<c<d<b<b+1.$$
 It follows from Claim \ref{blue} and Claim \ref{green} that $\phi(\delta_a,\delta_b)=\phi(\delta_a,\delta_c)=\phi(\delta_a,\delta_d)=\phi(\delta_c,\delta_d)=\text{red}$. (See Figure \ref{red-red}).

 \begin{figure}[H] 
\centering 
\tikzset{every picture/.style={line width=0.8pt}} 

\begin{tikzpicture}[x=0.75pt,y=0.75pt,yscale=-1,xscale=1]
\draw    (200,3570.5) -- (200,3419.5) ;
\draw [shift={(200,3417.5)}, rotate = 90] [color={rgb, 255:red, 0; green, 0; blue, 0 }  ][line width=0.75]    (10.93,-3.29) .. controls (6.95,-1.4) and (3.31,-0.3) .. (0,0) .. controls (3.31,0.3) and (6.95,1.4) .. (10.93,3.29)   ;

\draw  [fill={rgb, 255:red, 0; green, 0; blue, 0 }  ,fill opacity=1 ] (484.98,3497.5) .. controls (484.98,3495.83) and (486.33,3494.48) .. (488,3494.48) .. controls (489.67,3494.48) and (491.02,3495.83) .. (491.02,3497.5) .. controls (491.02,3499.17) and (489.67,3500.52) .. (488,3500.52) .. controls (486.33,3500.52) and (484.98,3499.17) .. (484.98,3497.5) -- cycle ;

\draw  [fill={rgb, 255:red, 0; green, 0; blue, 0 }  ,fill opacity=1 ] (459.5,3475.52) .. controls (459.5,3473.85) and (460.85,3472.5) .. (462.52,3472.5) .. controls (464.18,3472.5) and (465.53,3473.85) .. (465.53,3475.52) .. controls (465.53,3477.18) and (464.18,3478.53) .. (462.52,3478.53) .. controls (460.85,3478.53) and (459.5,3477.18) .. (459.5,3475.52) -- cycle ;

\draw  [fill={rgb, 255:red, 0; green, 0; blue, 0 }  ,fill opacity=1 ] (224.5,3469.52) .. controls (224.5,3467.85) and (225.85,3466.5) .. (227.52,3466.5) .. controls (229.18,3466.5) and (230.53,3467.85) .. (230.53,3469.52) .. controls (230.53,3471.18) and (229.18,3472.53) .. (227.52,3472.53) .. controls (225.85,3472.53) and (224.5,3471.18) .. (224.5,3469.52) -- cycle ;

\draw  [dash pattern={on 4.5pt off 4.5pt}]  (227.52,3469.52) -- (249.52,3449.52) ;

\draw    (200,3570.5) -- (517,3572.5) ;

\draw [color={rgb, 255:red, 208; green, 2; blue, 27 }  ,draw opacity=1 ] [dash pattern={on 4.5pt off 4.5pt}]  (249.52,3449.52) -- (462.52,3475.52) ;

\draw  [dash pattern={on 4.5pt off 4.5pt}]  (462.52,3475.52) -- (488,3497.5) ;

\draw  [dash pattern={on 4.5pt off 4.5pt}]  (462.52,3475.52) -- (300.52,3530.52) ;

\draw [color={rgb, 255:red, 208; green, 2; blue, 27 }  ,draw opacity=1 ] [dash pattern={on 4.5pt off 4.5pt}]  (249.52,3449.52) -- (300.52,3530.52) ;

\draw [color={rgb, 255:red, 208; green, 2; blue, 27 }  ,draw opacity=1 ] [dash pattern={on 4.5pt off 4.5pt}]  (249.52,3449.52) -- (373.52,3554.52) ;

\draw [color={rgb, 255:red, 208; green, 2; blue, 27 }  ,draw opacity=1 ] [dash pattern={on 4.5pt off 4.5pt}]  (300.52,3530.52) -- (373.52,3554.52) ;

\draw  [fill={rgb, 255:red, 0; green, 0; blue, 0 }  ,fill opacity=1 ] (246.5,3449.52) .. controls (246.5,3447.85) and (247.85,3446.5) .. (249.52,3446.5) .. controls (251.18,3446.5) and (252.53,3447.85) .. (252.53,3449.52) .. controls (252.53,3451.18) and (251.18,3452.53) .. (249.52,3452.53) .. controls (247.85,3452.53) and (246.5,3451.18) .. (246.5,3449.52) -- cycle ;
 
\draw  [fill={rgb, 255:red, 0; green, 0; blue, 0 }  ,fill opacity=1 ] (370.5,3554.52) .. controls (370.5,3552.85) and (371.85,3551.5) .. (373.52,3551.5) .. controls (375.18,3551.5) and (376.53,3552.85) .. (376.53,3554.52) .. controls (376.53,3556.18) and (375.18,3557.53) .. (373.52,3557.53) .. controls (371.85,3557.53) and (370.5,3556.18) .. (370.5,3554.52) -- cycle ;

\draw  [fill={rgb, 255:red, 0; green, 0; blue, 0 }  ,fill opacity=1 ] (297.5,3530.52) .. controls (297.5,3528.85) and (298.85,3527.5) .. (300.52,3527.5) .. controls (302.18,3527.5) and (303.53,3528.85) .. (303.53,3530.52) .. controls (303.53,3532.18) and (302.18,3533.53) .. (300.52,3533.53) .. controls (298.85,3533.53) and (297.5,3532.18) .. (297.5,3530.52) -- cycle ;

\draw (242.65,3428.5) node [anchor=north west][inner sep=0.75pt]  [font=\small]  {$\delta _{a}$};

\draw (210.65,3467.5) node [anchor=north west][inner sep=0.75pt]  [font=\small]  {$\delta _{a-1}$};

\draw (456.65,3455.5) node [anchor=north west][inner sep=0.75pt]  [font=\small]  {$\delta _{b}$};

\draw (294.65,3508.5) node [anchor=north west][inner sep=0.75pt]  [font=\small]  {$\delta _{c}$};

\draw (380.65,3547.5) node [anchor=north west][inner sep=0.75pt]  [font=\small]  {$\delta _{d}$};

\draw (490.65,3492.5) node [anchor=north west][inner sep=0.75pt]  [font=\small]  {$\delta _{b+1}$};

\end{tikzpicture}

\caption{ The process of finding $K_5^{(4)}$ based on Claim \ref{blue} and Claim \ref{green}. }
\label{red-red}
\end{figure}
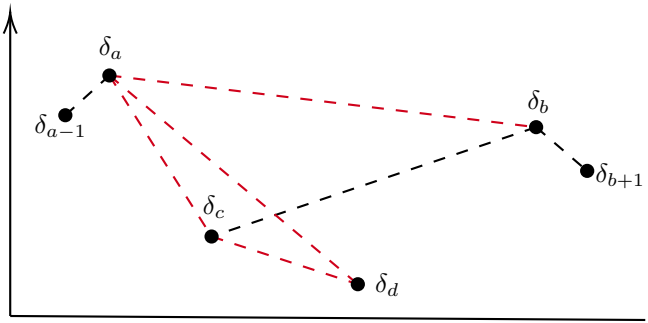

 From this, $\{v_{a-1},v_{a},v_{c},v_{c+1},v_{b+1}\}$ forms a $K_5^{(4)}$ by applying \textbf{($\Gamma_3$)} and \textbf{($\Lambda$)}, a contradiction. Hence, $\phi(\delta_c,\delta_b)=\text{red}$. Similarly, $\phi(\delta_d,\delta_b)=\text{red}$. Thus, $\{v_{c},v_{d},v_{d+1},v_{b+1},v_{b+2}\}$ forms a $K_5^{(4)}$ by \textbf{($\Gamma_1$)} and  \textbf{($\Lambda$)}, which is a contradiction.

This completes the proof of $\alpha_5(H)<2^{11}n^{11}+1$.

\section{Proof of Corollary \ref{main-results-2}}\label{stepping up}
In this section, we prove Corollary \ref{main-results-2} by applying a variant of the Erd\H{o}s-Hajnal stepping-up lemma. A closely related stepping-up lemma appears in \cite[Theorem 4.5 and Theorem 4.6]{F-H-L-L}, and its original proof is rather technical. To ensure the completeness of our own argument within this paper, we give a simpler lemma below, which is slightly weaker but suffices for our application.

For any set $S$, recall that $m(S)$ and $n(S)$ denote the number of local extrema and local monotones in the sequence $\delta(S)$, respectively. We have the following stepping-up lemma.

\begin{lemma}[Stepping-up Lemma]\label{step}
    For $k\ge 5$ and $s\ge k+7$, 
    \begin{align}\label{step-formula}
        f^{(k)}_{k+1,s}(2^N)<4(f^{(k-1)}_{k,s-1}(N))^2.
    \end{align}
\end{lemma}
Since $f^{(k)}_{k+1,s}(N)\ge (\log_{(k-2)}N)^{\Omega(1)}$ for all $k\ge3$ and $s\ge k+2$ by (\ref{low-bound}), the lower bound stated in Corollary \ref{main-results-2} is immediate. Thus, it suffices to prove the upper bound, which follows immediately from the above stepping-up lemma by noting that 
\begin{align*}
    f^{(k)}_{k+1,s}(N)\le f^{(k)}_{k+1,s}(2^{\lceil \log N\rceil})\overset{(\ref{step-formula})}{<}4(f^{(k-1)}_{k,s-1}(\lceil \log N\rceil))^2
    \le (\log_{(k-2)}N)^{O(1)}
\end{align*}
for all $k\ge 5$ and $s\ge k+7$. Now, we give a proof of Lemma \ref{step}.
\medskip

\noindent\textit{Proof of Lemma \ref{step}.~}From the definition of $f^{(k-1)}_{k,s-1}(N)$, there is a $K^{(k-1)}_{s-1}$-free $(k-1)$-graph $G'$ on $\{0,1,\ldots,N-1\}$ with $\alpha_{k}(G')=f^{(k-1)}_{k,s-1}(N)$. We use $G'$ to construct a $K^{(k)}_{s}$-free $k$-graph $G$ on $V (G)=\{0, 1,\ldots, 2^{N}-1\}$ with $\alpha_{k+1}(G) < 4(\alpha_{k}(G'))^2$.  We
define the edges of $G$ as follows. For any $k$-tuple $e=\langle v_1,\ldots, v_k\rangle$ of $V(G)$, set $e\in E(G)$ if and
only if one of the following holds:
\begin{enumerate}
    \item[(i)] $m(e) = 0$ (i.e., $\delta_1,\ldots, \delta_{k-1}$ form a monotone sequence) and $(\delta_1,\ldots, \delta_{k-1})\in E(G')$.

    \item[(ii)] $m(e) = k-3$.

    Furthermore, if $k\ge 6$, then
     
    \item[(iii)] $m(e) =k-4$.
\end{enumerate}

We have the following property.
\medskip

\textbf{Property VI.} For any $v_1 < \cdots< v_r$, suppose that $\delta_1,\ldots, \delta_{r-1}$ form a monotone sequence. If every $k$-tuple in $\langle v_1,\ldots,v_r\rangle$ is in $E(G)$, then every $(k-1)$-tuple in $\{\delta_1,\ldots, \delta_{r-1}\}$
is in $E(G')$.
\medskip

First, we show that $\alpha_{k+1}(G) < 4(\alpha_{k}(G'))^2$.
Set $m = 4n^2$, where $n = \alpha_{k}(G')$. Assume for contradiction that there exist vertices $v_1 <\dots < v_m$ that induce
a $(k+1)$-independent set in $G$. If the corresponding sequence $\delta_1,\ldots, \delta_{m-1}$ contains fewer than $4n$ local extrema, then there exist some $j$ such that $\{\delta_p\}_{p=j}^{j+n}$ form a monotone sequence. Without loss of generality, we assume
that this sequence is monotone decreasing. Since $n=\alpha_{k}(G')$,  $\{\delta_p\}_{p=j}^{j+n}$ contain a copy of $K^{(k-1)}_k$ in $G'$, say this copy is given by $\delta_{j_1},\ldots,\delta_{j_k}$, then by Property VI, $\{ v_{j_1},\ldots,v_{j_k},v_{{j_k}+1}\}$ forms a $K^{(k)}_{k+1}$ in $G$, a contradiction. Thus we may assume that the sequence $\delta_1,\ldots, \delta_{m-1}$ contains at least $4n$ local extrema. 
Consequently, we can find a subsequence of consecutive local extrema $\delta_{j_1}>\delta_{j_2}<\delta_{j_3}>\delta_{j_4}<\delta_{j_5}>\cdots<\delta_{j_{2i-1}}>\delta_{j_{2i}}<\cdots\delta_{j_{k}}$, where $\delta_{j_{2i}}$ are local minima and $\delta_{j_{2i+1}}$ are local maxima for all $i\in \{1,2,\ldots, \lfloor k/2 \rfloor-1\}$. 
(In particular, if $k=5$, we can find the above sequence such that $\delta_{j_3}>\delta_{j_i}$ for $i\in \{1,5\}$ when $\delta_{j_3}$ is chosen as a local maximum among all local maxima in the sequence. Otherwise, there would be a monotone sequence of length $n+1$ among the local maxima, which, by the similar argument as above, yields a contradiction with Property~VI.) 
Thus, we can check that all $k$-tuples in $\{v_{j_1},v_{j_2},v_{j_2+1},\dots,v_{j_{2i}},v_{j_{2i}+1}  \ldots,v_{j_{k}+1}\}$ are edges by (ii) or (iii), which form a copy of $K_{k+1}^{(k)}$ in $G$, a contradiction. 

Suppose, to the contrary, that there is a set $P=\langle v_1,\ldots,v_{s} \rangle$ that induces a $K^{(k)}_{s}$ in $G$. If $m(P)=0$, then we can find a $K_{s-1}^{(k-1)}$ in $G'$ due to (i) and Property VI, a contradiction. Let $\delta_i=\delta(v_i,v_{i+1})$ for all $i\in [s-1]$. 
We now turn to the case $m(P)\ge1$ and we need the following observation.

\begin{observation} \label{k-obser}
    For $k\ge 6$, there is no $\ell\in[s-4]$ such that $\{\delta_\ell,\delta_{\ell+1},\delta_{\ell+2},\delta_{\ell+3}\}$ is monotone. Moreover, if $k=5$, then there is no $\ell\in[s-3]$ such that $\{\delta_\ell,\delta_{\ell+1},\delta_{\ell+2}\}$ is monotone.
\end{observation}
\noindent\textit{Proof of Observation \ref{k-obser}.~}We first consider $k\ge6$. Suppose such an $\ell$ exists. Consider the longest consecutive monotone sequence within $\delta(P)$ that contains $\{\delta_\ell,\delta_{\ell+1},\delta_{\ell+2},\delta_{\ell+3}\}$; denote it by $\{\delta_i,\delta_{i+1},\ldots,\delta_{j-1},\delta_j\} $. By definition, we have $i\le \ell$ and $j\ge \ell+3$. Since $m(P)\ge 1$, $\{\delta_i,\delta_{i+1},\ldots,\delta_{j-1},\delta_j\}\neq \delta(P)$. Therefore, $\delta_i$ or $\delta_j$ is a local extremum. Without loss of generality, we assume that $\delta_{j}$ is a local extremum. If $j\ge k-2$, then $\{v_{j-k+3},v_{j-k+4},\ldots, v_j,v_{j+1},v_{j+2}\}$ cannot be an edge, because $1\le m(\{v_{j-k+3},v_{j-k+4},\ldots, v_j,v_{j+1},v_{j+2}\})\le k-5$, which is a contradiction. If $j<k-2$, then $\{v_1,\ldots,v_j,v_{j+1}, v_{j+2}, \ldots,v_k\}$ is likewise not an edge because $1\le m(\{v_1,\ldots,v_j,v_{j+1},v_{j+2}, \ldots,v_k\})\le k-5$, again yielding a contradiction. 

If $k=5$ and such an $\ell$ exists, then we can use a similar argument as above to extend $\{\delta_\ell,\delta_{\ell+1},\delta_{\ell+2}\}$ to a $5$-sequence $f$ such that $m(f)=1$. Consequently, $f$ is not an edge, which yields a contradiction.\hfill$\Box$

For $k=5$, it follows from Observation \ref{k-obser} that $\{\delta_i\}_{i=2}^{s-2}$ are all local extrema.
Thus, we can choose two consecutive local maxima $\delta_{m_1}$ and $\delta_{m_2}$, then $\delta_{m_1}\neq\delta_{m_2}$ from Property I and II.
Without loss of generality, we assume that $\delta_{m_1}<\delta_{m_2}$, now $\{v_{{m_1}-1},v_{m_1},v_{{m_1}+1},v_{{m_2}+1},v_{{m_2}+2}\}$ is not an edge since $\delta(v_{{m_1}-1},v_{m_1},v_{{m_1}+1},v_{{m_2}+1},v_{{m_2}+2})=(\delta_{{m_1}-1}<\delta_{m_1}<\delta_{m_2}>\delta_{{m_2}+1})$, a contradiction.

In the following, we can assume that $k\ge6$ and we will use the following claim to derive a contradiction.

\begin{claim}\label{k-claim}
    There exists a subsequence $S\subset\{v_i\}_{i=1}^{13}$ with $|S|=6$ such that $m(S)=1$ and $n(S)=2$.
\end{claim}
\noindent\textit{Proof of Claim \ref{k-claim}.~}Let $\delta_{j_1}$ denote the unique largest element in $\{\delta_i\}_{i=1}^{12}$. Without loss of generality, we assume that $j_1\in[6]$. 
Let $\delta_{j_2}$ denote the unique largest element in $\{\delta_i\}_{i=j_1+1}^{12}$, then $\delta_{j_1}>\delta_{j_2}$.

If there exists a local monotone $\delta_{\ell}$ between $\delta_{j_1}$ and $\delta_{j_2}$, then without loss of generality we assume that $\delta_{\ell-1}<\delta_{\ell}<\delta_{\ell+1}$ and $j_1<\ell-1<\ell+1<j_2$.
Since $\delta(v_{j_1},v_{\ell-1},v_{\ell},v_{\ell+1},v_{\ell+2},v_{j_2+1})=(\delta_{j_1}>\delta_{\ell-1}<\delta_{\ell}<\delta_{\ell+1}<\delta_{j_2})$, we can take $S=\{v_{j_1},v_{\ell-1},v_{\ell},v_{\ell+1},v_{\ell+2},v_{j_2+1}\}$ to satisfy $m(S)=1$ and $n(S)=2$.  If there exists a local minimum  $\delta_{\ell}$ between $\delta_{j_1}$ and $\delta_{j_2}$, that is, $\delta_{\ell-1}>\delta_{\ell}<\delta_{\ell+1}$ and $j_1<\ell-1<\ell+1<j_2$, then $S=\{v_{j_1},v_{\ell-1},v_{\ell},v_{\ell+1},v_{\ell+2},v_{j_2+1}\}$ satisfies the requirement, noting that $\delta(v_{j_1},v_{\ell-1},v_{\ell},v_{\ell+1},v_{\ell+2},v_{j_2+1})=(\delta_{j_1}>\delta_{\ell-1}>\delta_{\ell}<\delta_{\ell+1}<\delta_{j_2})$. Thus, it is enough to restrict attention to  $j_2-j_1\le4$. Otherwise, $j_2-j_1\ge5$, then there exists either a local monotone or a local minimum. In particular, if $j_1=1$, then $j_2\le 5$. Recall that there is no $\ell\in[6,9]$ such that $\{\delta_\ell,\delta_{\ell+1},\delta_{\ell+2},\delta_{\ell+3}\}$ is monotone from Observation \ref{k-obser}. 
Thus, there exists $\ell\in[7,11]$ such that $\delta_{\ell-1}>\delta_\ell<\delta_{\ell+1}$. For this $\ell$, the set $S=\{v_{j_1},v_{j_2},v_{\ell-1},v_{\ell},v_{\ell+1},v_{\ell+2}\}$ yields the desired configuration.

Hence, we only need to consider $j_2\le j_1+4$, where $j_1\in [2,6]$. We can immediately conclude that $j_2\le 10$. Let us first consider the case where $j_2\le 9$. Note that there exists $\ell\in[j_2+1,j_2+2]$ such that $\delta_{\ell}>\delta_{\ell+1}$. Otherwise, we have $\delta(v_{j_1},v_{j_2},v_{j_2+1},v_{j_2+2},v_{j_2+3},v_{j_2+4})=(\delta_{j_1}>\delta_{j_2}>\delta_{j_2+1}<\delta_{j_2+2}<\delta_{j_2+3})$, and $S=\{v_{j_1},v_{j_2},v_{j_2+1},v_{j_2+2},v_{j_2+3},v_{j_2+4}\}$ would satisfy the requirement. For this $\ell$, we have $\delta(v_{j_1-1},v_{j_1},v_{j_2},v_{\ell},v_{\ell+1},v_{\ell+2})=(\delta_{j_1-1}<\delta_{j_1}>\delta_{j_2}>\delta_\ell>\delta_{\ell+1})$, and we can take $S=\{v_{j_1-1},v_{j_1},v_{j_2},v_{\ell},v_{\ell+1},v_{\ell+2}\}$, where we note $j_1\ge 2$.
Then, we consider the case where $j_2= 10 $, it follows that $j_1=6$. We may assume that $\delta_{\ell}<\delta_{\ell+1}$ for some $\ell\in[3]$. Otherwise, we can take $S=\{v_1,v_2,v_3,v_4,v_5,v_{7}\}$ since $\delta(v_1,v_2,v_3,v_4,v_5,v_{7})=(\delta_1>\delta_2>\delta_3>\delta_4<\delta_{6})$. For this $\ell$, we have $\delta(v_\ell,v_{\ell+1},v_{\ell+2}, v_{7},v_{12},v_{13})=(\delta_{\ell}<\delta_{\ell+1}<\delta_{6}>\delta_{10}>\delta_{12})$, and we can take $S=\{v_\ell,v_{\ell+1},v_{\ell+2}, v_{7},v_{12},v_{13}\}$.
\hfill$\Box$
\medskip

For $k\ge 6$, take $S=\{u_1,\ldots,u_6\}$ from Claim \ref{k-claim}. Then $\{u_1,\ldots,u_6,v_{14},\ldots,v_{k+7}\}$ cannot be an edge because $$1\le m(\{u_1,\ldots,u_6,v_{14},\ldots,v_{k+7}\})\le k-5.$$ This yields a contradiction. 

This completes the proof of Lemma \ref{step}. \hfill$\Box$

\section{Concluding Remarks}
In this paper, we prove $f^{(4)}_{5,s}(N)= (\log \log N)^{\Theta(1)}$ for every  $s\ge 11$
by applying the stepping-up technique to a graph provided by a new probabilistic
construction. 
This partially confirms the conjecture of Mubayi and Suk~\cite{M-S-5},  but extending it to the remaining values $s\in[6,10]$ would likely require additional conditions for the 2-graph and, more importantly, new ideas.

On the other hand, we also consider a related question for a non-complete forbidden configuration. Given a $k$-graph $F$, an \textit{$F$-independent set} in a $k$-graph $H$ is a vertex subset $S$ such that $H[S]$ contains no copy of $F$. The \textit{generalized Erd\H{o}s-Rogers function} $f^{(k)}_{F,s}(N)$ denotes the largest $m$ such that every $K_s^{(k)}$-free $k$-graph on $N$ vertices contains an $F$-free induced subgraph on $m$ vertices. Here $F$-free means that it does not contain $F$ as a subgraph. Namely, $$f^{(k)}_{F,s}(N)=\min\{\alpha_F(H): |V(H)|=N~\text{and}~H~\text{is}~K_s^{(k)}\text{-free}\}.$$

Recall that the off-diagonal Ramsey number $r_k(s,n)$ is the minimum $N$ such that every red/blue coloring of the edges of $K_N^{(k)}$ results in a monochromatic red copy of $K^{(k)}_s$ or a monochromatic blue copy of $K^{(k)}_n$, where $k,s$ are fixed and $n$ tends to infinity. Let $5^-=K^{(4)}_5-e$ denote the hypergraph obtained from $K^{(4)}_5$ by removing one edge. We are also interested in determining the magnitude of $f^{(4)}_{5^-,s}(N)$ for every  $s\ge 5$. Since $r_4(s,n)\le 2^{2^{n^c}}$ for  every  $s\ge 5$ from \cite{E-R-2}, we have $f^{(4)}_{5^-,s}(N)\ge f^{(4)}_{4,s}(N)\ge (\log\log N)^{\Omega(1)}$ by noting (\ref{R-EG}). Thus, (\ref{con-2}) suggests $f^{(4)}_{5^-,s}(N)=(\log\log N)^{\Theta(1)}$ for  every $s\ge 6$. Theorem \ref{main-results} verifies this for every  $s\ge 11$, and in \cite{D-H-L-W-2} we use the framework of this paper to prove it for every $s\ge 6$. 
More importantly, refining the methods introduced in this paper could help settle whether $f^{(4)}_{5^-,5}(N)=(\log\log N)^{O(1)}$ holds. A positive answer would confirm the conjectured double-exponential lower bound on $r_4(5,n)$.

\end{spacing}
\end{document}